\documentclass[12pt]{amsart}

\usepackage{amssymb,latexsym,amsmath,extarrows, mathrsfs, amsthm}
\usepackage{graphicx}

\usepackage[margin=1.2in, centering]{geometry}

\usepackage{color}
\usepackage{wasysym}

\usepackage{cite}

\usepackage{hyperref}

\usepackage{enumerate}
\usepackage{mathtools}

\usepackage{enumitem}

\DeclareMathOperator{\diam}{diam}

\usepackage{float}

\newtheorem{theorem}{Theorem}[section]
\newtheorem{lemma}[theorem]{Lemma}

\newtheorem{definition}[theorem]{Definition}
\newtheorem{corollary}[theorem]{Corollary}

\theoremstyle{remark}
\newtheorem{remark}{\bf  \itshape  Remark}

\newcommand{\ZR}{\mathbb{R}}
\newcommand{\ZZ}{\mathbb{Z}}

\newcommand{\ZH}{\mathbb{H}}

\newcommand{\PSL}{\mathrm{PSL}}

\allowdisplaybreaks

\title{Distinct distances on hyperbolic surfaces}

\author{Xianchang Meng}

\address{Mathematisches Institut,
	Georg-August Universit\"{a}t G\"{o}ttingen,
	Bunsenstra{\ss}e 3-5,
	D-37073 G\"{o}ttingen,
	Germany}
\email{xianchang.meng@uni-goettingen.de}


\begin{document}
	
	\begin{abstract}
		For any cofinite Fuchsian group $\Gamma\subset{\rm PSL}(2, \mathbb{R})$, we show that any set of $N$ points on the hyperbolic surface $\Gamma\backslash\mathbb{H}^2$ determines $\geq C_{\Gamma} \frac{N}{\log N}$ distinct distances for some constant $C_{\Gamma}>0$ depending only on $\Gamma$. In particular, for $\Gamma$ being any finite index subgroup of ${\rm PSL}(2, \mathbb{Z})$ with  $\mu=[{\rm PSL}(2, \mathbb{Z}): \Gamma ]<\infty$,  any set of $N$ points  on $\Gamma\backslash\mathbb{H}^2$ determines $\geq C\frac{N}{\mu\log N}$ distinct distances for some absolute constant $C>0$.

	\end{abstract}
	
\maketitle

\section{Introduction}	

Erd\H{o}s \cite{Erdos} in 1946 asked the question of finding  the minimal number of distinct distances among any $N$ points in the  plane.  The breakthrough work of Guth-Katz \cite{Guth-Katz}  gave the lower bound $\geq C\frac{N}{\log N}$ for some constant $C>0$ in the Euclidean plane, which is sharp up to a factor of $\log$. Another related and widely studied conjecture is
the Falconer's conjecture which asks about the lower bound of the Hausdorff dimension of the sets in $\mathbb{R}^d$ for which the difference set has positive Lebesgue measure.  The Falconer's conjecture can be viewed as a  continuous analogue of the distinct distances problem.  Interested readers may check Falconer \cite{Falconer}, Guth-Iosevich-Ou-Wang \cite{Guth-Iosevich-Ou-Wang}, Iosevich \cite{Iosevich} etc. 
The Erd\H{o}s-Falconer type problems have been generalized to  other spaces  and applied to certain sum-product estimates, see e.g.  Bourgain-Tao \cite{Bourgain-Tao},   Hart-Iosevich-Koh-Rudnev \cite{HIKR}, Roche-Newton and Rudnev \cite{RocheNewton-Rudnev},  Rudnev-Selig \cite{Rudnev-Selig}, Sheffer-Zahl \cite{Sheffer-Zahl}, and blog of Tao \cite{Tao} etc.
  However, the distinct distances problem has not been considered in  hyperbolic surfaces until very recently by Lu and the author  in \cite{Lu-Meng} where the modular surface and hyperbolic surfaces  with cocompact fundamental groups are studied.  But this problem is still open for  more general hyperbolic surfaces arising from  non-cocompact Fuchsian groups. 

In this paper, for all cofinite Fuchsian groups $\Gamma$,  we give complete answer to the distinct distances problem for all hyperbolic surfaces $\Gamma\backslash\ZH^2$ endowed with the hyperbolic metric from $\ZH^2$.

\begin{theorem}\label{thm-cofinite}
	For any cofinite Fuchsian group $\Gamma\subset \PSL(2, \ZR)$, any set of $N$ points on the hyperbolic surface $\Gamma\backslash\ZH^2$ determines $\geq C_{\Gamma}\frac{N}{\log N}$ distinct distances for some constant $C_{\Gamma}$ depending only on  $\Gamma$. 
\end{theorem}

In particular, for finite index subgroups of the modular group $\PSL(2, \ZZ)$, we extract out the dependence of the implied constants on the index.

\begin{theorem}\label{thm-finite-index}
	For any finite index subgroup $\Gamma$ of $\PSL(2, \ZZ)$ with $[ \PSL(2,\ZZ):\Gamma ]=\mu$, any set of $N$ points on the hyperbolic surface  $\Gamma\backslash\ZH^2$ determines $\geq C \frac{N}{\mu\log N}$ distinct distances for some absolute constant $C>0$. 
\end{theorem}

Theorem \ref{thm-finite-index} has application to equilateral dimension problem. The equilateral dimension of a metric space is the maximal number of points in the space with pairwise equal distance. It has been studied in various spaces, see Alon-Milman \cite{Alon-Milman}, Guy \cite{Guy}, Koolen \cite{Koolen} etc. For instance, the equilateral dimension of the $n$-dimensional Euclidean space is $n+1$.  However, we are not aware of any result in literature about the equilateral dimension of general hyperbolic surfaces. We observe that the lower bound in Theorem \ref{thm-finite-index} is not trivial for distinct distances among any set of size $N\gg \mu^{1+\epsilon}$.  Thus the following corollary holds.
\begin{corollary}
	For any subgroup $\Gamma$ of $\PSL(2, \ZZ)$ with finite index $[ \PSL(2,\ZZ):\Gamma ]=\mu$, the equilateral dimension of the hyperbolic surface $\Gamma\backslash\ZH^2$ is $\ll \mu^{1+\epsilon}$ for any $ \epsilon>0$. 
\end{corollary}

The isometry group of the hyperbolic plane $\ZH^2$ is  $\PSL(2, \ZR)$   which acts on $\ZH^2$ by M\"{o}bius transformation:
\[z\mapsto \gamma(z):=\frac{az+b}{cz+d}, \text{ for }\gamma=\begin{pmatrix}a&b\\c&d\end{pmatrix}\in\PSL(2, \ZR), z\in\ZH^2.\] 
For any discrete  subgroup $\Gamma$ of $\PSL(2, \ZR)$, i.e. a \textit{Fuchsian group}, the distance between any two points $p, q $ on the hyperbolic surface $ Y\cong \Gamma\backslash\ZH^2$ is 
$$d_Y(p, q):=\min_{\gamma_1, \gamma_2\in \Gamma} d_{\ZH^2}(\gamma_1(p), \gamma_2(q))=\min_{\gamma_1, \gamma_2\in\Gamma} d_{\ZH^2}(p, \gamma_1^{-1}\gamma_2(q) )=\min_{\gamma\in\Gamma}d_{\ZH^2}(p, \gamma(q)).$$
\begin{figure}[ht]
	\centering
	\includegraphics[width=0.6\textwidth]{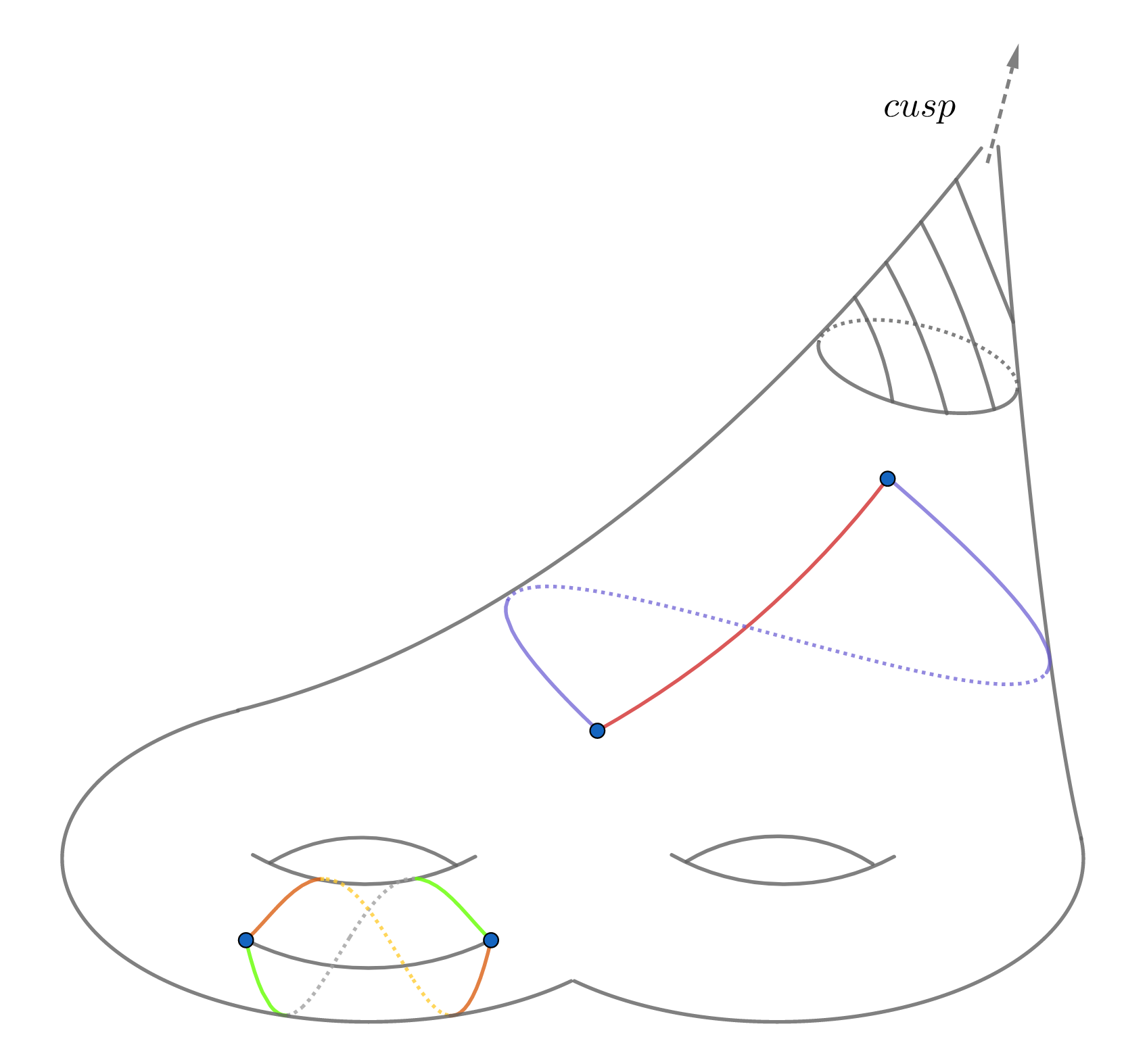}
	\caption{Distances on hyperbolic surface}
\end{figure}

Instead of calculating distances on the surface directly, we consider representatives of the points in a fundamental domain $F_{\Gamma}$ of $\Gamma$. In \cite{Lu-Meng}, Lu and the author introduced the concept of a ``geodesic cover" $\Gamma'\subset\Gamma$ such that for any $p, q\in F_{\Gamma}$, 
$$  d_Y(p, q)= d_{\ZH^2}(p, \gamma' q) ~\text{for some}~\gamma'\in\Gamma'.$$
If there exists a finite geodesic cover,  one can derive a lower bound for the distinct distance problem on the hyperbolic surface $\Gamma\backslash\ZH^2$.  For the modular surface $\PSL(2, \ZZ)\backslash\ZH^2$, we would be able to find a finite geodesic cover by working explicitly with matrices in $\PSL(2, \ZZ)$. However, it is hard to tackle general non-cocompact Fuchsian groups this way since we cannot explicitly write out all the elements. Another difficulty to find such a finite geodesic cover is, the number of representatives we need to examine would blow up if the fundamental domain has many inequivalent cusps. This is not an issue for modular surface which has only one inequivalent cusp, and that the imaginary parts of points in a fundamental domain of $\PSL(2, \ZZ)$ are all bounded below (or bounded above if we choose other type of fundamental domain). Therefore,  representatives we have to examine will not have very small imaginary parts and the number of them could be bounded.  But in the general case, if a pair of points are close to two inequivalent cusps respectively, the number of representatives we have to examine might lose control.

In order to overcome such difficulties, we propose a more general concept of a geodesic cover defined on any subset of a fundamental domain $F_{\Gamma}$ and also defined in different base groups, see Definition \ref{defn-geodesic-covering-number}. By building relations between geodesic covers of different subregions in $F_{\Gamma}$ and geodesic covers of certain regions in different groups, we  prove lower bounds for distinct distances on hyperbolic surfaces associated with any cofinite Fuchsian group. See Lemmas \ref{lem-geodesic-number-Distinct-distance}, \ref{lem-geodesic-cover-number-finite} and \ref{lem-finite-index-geo-cover-number} for details.

\bigskip
\textbf{Acknowledgment.} The author is partially supported by the Humboldt Professorship of Professor Harald Helfgott.

\section{Preliminaries and preparations}

First we briefly summarize the properties of Fuchsian groups (see Beardon \cite{Beardon} or Katok \cite{Katok} for more related materials).
A subgroup $\Gamma$ of $\PSL(2, \ZR)$ is a \textit{Fuchsian group} if and only if $\Gamma$ acts \textit{properly discontinuously} on $\ZH^2$. Thus the $\Gamma$-orbit of any point $z\in\ZH^2$ is\textit{ locally finite}, which means any compact set $K\subset\ZH^2$ contains only finite number of orbit points, i.e. the set $\Gamma z\cap K$ is finite for any $z\in\ZH^2$.

A \textit{cofinite} Fuchsian group is a discrete subgroup of $\PSL(2, \ZR)$ of finite covolume i.e. a fundamental domain of $\Gamma\backslash\ZH^2$ has finite hyperbolic area.  A cofinite discrete subgroup is  also called a \textit{lattice} in some other contexts. 
Siegel's theorem (see \cite{Katok}, Theorem 4.1.1) says cofinite Fuchsian group is \textit{geometrically finite}, i.e. there exists a convex fundamental domain with finitely many sides.

The cocompact Fuchsian groups has been considered in \cite{Lu-Meng}. In this paper, we focus on non-cocompact case. Suppose $\Gamma$ has parabolic elements, and thus its fundamental domain $F_{\Gamma}$ must have a vertex on $\hat{\ZR}$ which is called a \textit{cusp}. Since we assume $\Gamma$ is cofinite, by Siegel's theorem, its fundamental domain $F_{\Gamma}$ has finitely many cusps.

We use an idea of Iwaniec (see \cite{Iwaniec}, \S 2.2) to partition the fundamental domain of a Fuchsian group. Define the stability group as
$$\Gamma_z:=\{ \gamma\in\Gamma: \gamma z=z \}.$$
Given a cusp $\mathfrak{a}\in \hat{\ZR}$ for $\Gamma$. The stability group $\Gamma_{\mathfrak{a}}$ is a cyclic group generated by a parabolic element, say $\Gamma_{\mathfrak{a}}=\langle \gamma_{\mathfrak{a}}\rangle$. There exists $\sigma_{\mathfrak{a}}\in SL(2, \ZR)$ such that
\begin{equation}\label{eq-conjugate-to-translation}
\sigma_{\mathfrak{a}}\infty=\mathfrak{a}, \quad \sigma_{\mathfrak{a}}^{-1}\gamma_{\mathfrak{a}}\sigma_{\mathfrak{a}}=\begin{pmatrix}1&1\\ 0&1\end{pmatrix}.
\end{equation}
Then $\sigma_{\mathfrak{a}}^{-1}$ sends $\mathfrak{a}$ to $\infty$ and $\sigma_{\mathfrak{a}}$ maps the strip
\begin{equation}\label{infinite-part-of-fundamental-domain}
P(T):=\{ z=x+iy: 0<x<1, y\geq T\}.
\end{equation}
into the cuspidal zone
\begin{equation}\label{eq-F-cusp-to-F-infty}
F_{\mathfrak{a}, T}=\sigma_{\mathfrak{a}}P(T).
\end{equation}
The cuspidal zone $F_{\mathfrak{a}, T}$ is contained in a disc (the boundary is a horocycle) tangent to $\hat{\ZR}$ at $\mathfrak{a}$.

When there are more than one cusps, we may choose $T$ large enough such that the cuspidal zones are disjoint. By doing this, we divide the fundamental domain $F_{\Gamma}$ into cuspidal parts \begin{equation}\label{eq-partition-cuspidal-central}
F_{\bowtie,T}:=\bigcup_{\mathfrak{a}}F_{\mathfrak{a}, T}
\end{equation} and the central part $F_T:=F_{\Gamma}\setminus F_{\bowtie,T}$ (see Figure \ref{fig-cuspidal-partition}).

\begin{figure}[ht]
	\centering
	\includegraphics[width=.6\textwidth
	]{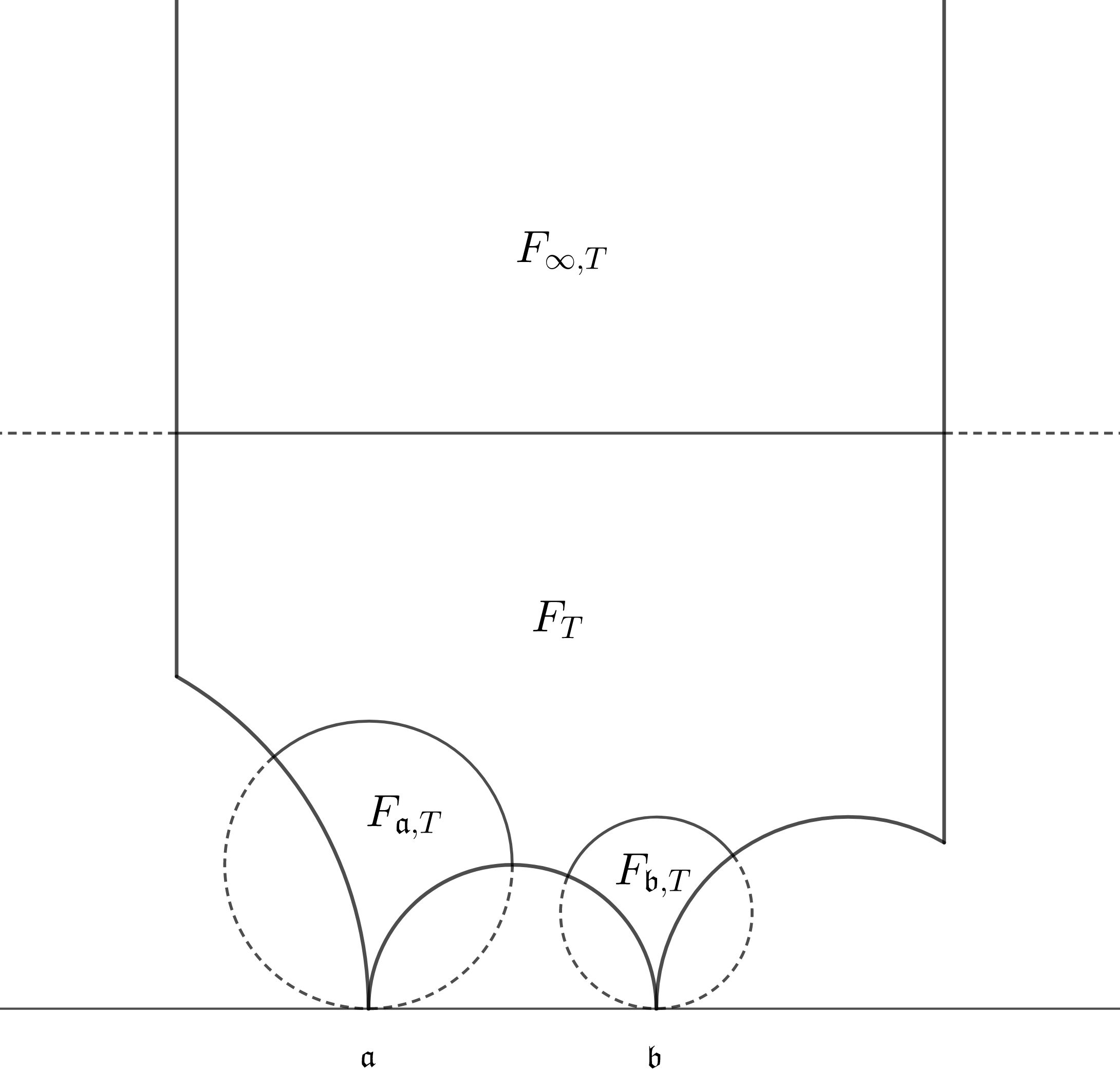}
	\caption{Cuspidal parts and central part}
	\label{fig-cuspidal-partition}
\end{figure}

Now we give the definition of a geodesic cover and geodesic-covering number of any region in a Fuchsian group. 
\begin{definition}\label{defn-geodesic-covering-number}
	Let $F_{\Gamma}$ be a fundamental domain of $\Gamma$, and $Y\cong\Gamma\backslash\ZH^2$ be the hyperbolic surface associated with $\Gamma$. For any subset $F'\subset F_{\Gamma}$, we say $\Gamma'\subset\Gamma$ is a \textbf{geodesic cover of $F'$ in $\Gamma$} if 
	\begin{equation}
	d_{Y}(p, q)=\min_{\gamma_1, \gamma_2\in\Gamma'} d_{\ZH^2}(\gamma_1(p), \gamma_2(q)), \forall p, q\in F'.
	\end{equation}
	
	We  call the smallest cardinality of $\Gamma'\subset\Gamma$ the \textbf{geodesic-covering number of $F'$  in $\Gamma$}, denoted by $K_{\Gamma}(F')$. 
	
\end{definition}

\begin{remark}
	A geodesic cover  always contains identity. If we take $F'=F_{\Gamma}$, this matches the definition of the geodesic cover in \cite{Lu-Meng}. 
\end{remark}
\begin{remark}
	Note that this definition depends on different regions and different base groups. We see that if $F''\subset F'\subset F_{\Gamma}$, then $K_{\Gamma}(F'')\leq K_{\Gamma}(F')$.  But for a subgroup $\Gamma^*$ of $\Gamma$, it is not clear if we have $K_{\Gamma^*}(F')\leq K_{\Gamma}(F')$ or vice versa. 
\end{remark}

Then we consider the geodesic-covering numbers of  the central part $F_T$ and cuspidal parts $F_{\mathfrak{a},T} $ for every cusp $\mathfrak{a}$. If all of them are finite, we are able to derive a lower bound for distinct distances on hyperbolic surfaces.  

\begin{lemma}\label{lem-geodesic-number-Distinct-distance}
Assume $\Gamma$ is a cofinite Fuchsian group  with a fundamental domain $F_{\Gamma}$. If the geodesic-covering numbers of $F_T$ and $F_{\mathfrak{a},T}$ for every cusp $\mathfrak{a}$ are all finite for some $T=T_{\Gamma}$, then any set of  $N$ points on the hyperbolic surface $\Gamma\backslash\ZH^2$ determines $\geq C_{\Gamma}\frac{N}{\log N}$ distinct distances for some constant $C_{\Gamma}$ depending on $\Gamma$. 
\end{lemma}

\begin{remark}
Throughout our proof, we assume the set concerned has no points lying on the boundary of $F_{\Gamma}$. If there are points lying on the boundary of $F_{\mathfrak{a}, T}$ for some cusp $\mathfrak{a}$, we may use a parabolic motion to map $F_{\mathfrak{a}, T}$ to a translate of $P(T)$ without points on the boundary.  If there are points lying on the boundary of  $F_T$, the same proof of Lemma \ref{lem-geodesic-cover-number-finite} also works for the closure of $F_T$. 
\end{remark}

\begin{remark}\label{remark-cocompact}
	If $\Gamma$ is a cocompact Fuchsian group, there is no cusp and $F_{\Gamma}$ is bounded.  Thus we only need to assume the cuspidal part is empty and the central part $F_{T}=F_{\Gamma}$ (for large enough $T$). In this case, the above lemma still holds. 
\end{remark}
\begin{proof}[Proof of Lemma \ref{lem-geodesic-number-Distinct-distance}]
Given a set $\mathcal{S}$ of $N$ points on the hyperbolic surface $Y\cong\Gamma\backslash\ZH^2$, we consider such $N$ points on a fundamental domain $F_{\Gamma}$. According to the partition \eqref{eq-partition-cuspidal-central} of 	$F_{\Gamma}$, either $F_{\bowtie,T}$ or $F_T$ has more than $N/2$ points on it. 

\textbf{Case 1).} If $F_T$ contains more than $N/2$ points, we only need to consider the lower bound for distinct distances among them, since this is also a lower bound for the $N$ points on the whole surface.


Denote the set of points on $F_T$ by $\mathcal{S}_1$. 
 Since we assume the geodesic-covering number of $F_T$ is finite, we choose a finite geodesic cover $\Gamma'\subset \Gamma$ with cardinality $|\Gamma'|=K_{\Gamma}(F_T)$. Define the distance set
$$d_{Y}(\mathcal{S}_1):=\{ d_{Y}(p, q): p, q\in \mathcal{S}_1  \}\subset \{ d_{\ZH^2}(p, q): p, q \in \cup_{\gamma\in \Gamma'} \gamma(\mathcal{S}_1)\},$$
and the distance quadruples
\begin{align}\label{eq-distance-quadruple-to-surface}
Q_{Y}(\mathcal{S}_1)&:=\{  (p_1, p_2; p_3, p_3)\in \mathcal{S}^4_1: d_{Y}(p_1, p_2)=d_{Y}(p_3, p_4)\neq 0 \}  \nonumber\\
&~\subset Q_{\ZH^2}\big(\cup_{\gamma\in\Gamma'}\gamma(\mathcal{S}_1) \big),
\end{align}
where 
\begin{equation}\label{eq-defn-QH}
Q_{\ZH^2}(\mathcal{P}):=\{ (p_1, p_2; p_3, p_4)\in \mathcal{P}^4: d_{\ZH^2}(p_1, p_2)=d_{\ZH^2}(p_3, p_4)\neq 0 \}.
\end{equation}

For any finite set of points $\mathcal{P}$ on a hyperbolic surface $Y$, the connection between $d_Y(\mathcal{P})$ and $Q_Y(\mathcal{P})$ is as follows. Suppose the elements of $d_Y(\mathcal{P})$
are $d_1, d_2, \cdots, d_k$ and $n_i$ is the number of pairs $(p_1, p_2)\in\mathcal{P}^2$ with distance $d_i$ ($1\leq i\leq k$). By the Cauchy-Schwarz inequality, we get
\begin{equation}
{|\mathcal{P}|\choose 2 }^2=\bigg( \sum_{i=1}^k n_i \bigg)^2\leq \bigg(\sum_{i=1}^k n^2_i \bigg)k=|Q_Y(\mathcal{P})| |d_Y(\mathcal{P})|,
\end{equation}
thus
\begin{equation}\label{Quadruple-to-distance}
|d_Y(\mathcal{P})|\geq \frac{ (|\mathcal{P}|^2-|\mathcal{P}| )^2  }{  |Q_Y(\mathcal{P})|}. 
\end{equation}

For any set of points $\mathcal{P}$ in $\ZH^2$, by an argument of Tao in his blog \cite{Tao} (see also \cite{Rudnev-Selig}), one can derive
\begin{equation}\label{eq-Tao-quaruple}
|Q_{\ZH^2}(\mathcal{P})|\ll |\mathcal{P} |^3 \log (|\mathcal{P}|).
\end{equation}
Recently, Lu-Meng \cite{Lu-Meng} also gave a different proof for the above estimate by modifying the framework of Guth-Katz and working explicitly with isometries of $\ZH^2$. 
Since the geodesic-covering number $K_{\Gamma}(F_T)$ of $F_T$ in $\Gamma$ is finite, the cardinality of $\cup_{\gamma\in\Gamma'}\gamma(\mathcal{S}_1)$  is  $\leq K_{\Gamma}(F_T) |\mathcal{S}_1	| \leq  K_{\Gamma}(F_T) N$.  By \eqref{eq-distance-quadruple-to-surface}, we derive that
\begin{equation}
|Q_Y(\mathcal{S}_1)|\ll K^3_{\Gamma}(F_T) N^3 (\log (K_{\Gamma}(F_T))+\log N).
\end{equation}
 Thus by \eqref{Quadruple-to-distance}, we get
 \begin{equation}
 |d_Y(\mathcal{S})|\geq |d_Y(\mathcal{S}_1)|\gg \frac{N}{ K^3_{\Gamma}(F_T)  (\log (K_{\Gamma}(F_T))+\log N) }\geq C'_{\Gamma}\frac{N}{\log N},
 \end{equation}
for some constant $C'_{\Gamma}>0$ depending  on $\Gamma$. 	
	
\textbf{Case 2).} There are more than $N/2$ points on $F_{\bowtie,T}$. Let $n_c<\infty$ be the number of cusps for the fundamental domain $F_{\Gamma}$. Then there exists one cusp $\mathfrak{b}$ such that $F_{\mathfrak{b},T}$ contains more than $N/2n_c$ points. We may assume all these points lie in the interior of $F_{\mathfrak{b}, T}$. Denote the set of points on $F_{\mathfrak{b},T}$ by $\mathcal{S}_2$. By a similar argument as in Case 1), we deduce that
	\begin{equation}
	|d_Y(\mathcal{S})|\geq |d_Y(\mathcal{S}_2)|\gg \frac{N/n_c}{ K^3_{\Gamma}(F_{\mathfrak{b}, T}) (\log(K_{\Gamma}(F_{\mathfrak{b},T}) ) +\log N  )  }  \geq C''_{\Gamma}\frac{N}{\log N},
	\end{equation}
	for some constant $C''_{\Gamma}>0$ depending  on $\Gamma$.
	
	Combining the cases 1) and 2), we finish the proof.
\end{proof}

\section{ Geodesic-covering numbers for cofinite Fuchsian groups }

In this section, we give the proof of Theorem \ref{thm-cofinite} based on Lemma \ref{lem-geodesic-number-Distinct-distance}. We only need to bound the geodesic-covering numbers of  $F_T$ and $F_{\mathfrak{a}, T} $ for every cusp $ \mathfrak{a}$. 
\begin{lemma}\label{lem-geodesic-cover-number-finite}
	Assume $\Gamma$ is a  cofinite Fuchsian group with a fundamental domain $F_{\Gamma}$. If we partition $F_{\Gamma}$ as in \eqref{eq-partition-cuspidal-central} for some large enough $T$ depending on $\Gamma$, the geodesic-covering numbers of $F_T$ and $F_{\mathfrak{a}, T}$ for every cusp $\mathfrak{a}$ in $\Gamma$ are all finite.  This is also true for the closure of $F_T$, i.e. $K_{\Gamma}(\overline{F}_T)<\infty$. 
\end{lemma}

\begin{proof}
We need to know the basic shape of a fundamental domain for any  fuchsian group. A convenient choice for us is \textit{Ford  domain} which was first introduced by L. R. Ford \cite{Ford}. It is known that Ford domain is a fundamental domain (see \cite{Katok}, Theorem 3.3.5).  
There are concrete methods to construct fundamental domains of Fuchsian groups, interested readers  may check Voight \cite{Voight} for an algorithmic method, and Kulkarni \cite{Kul} for construction of special polygons (also a fundamental domain)  for subgroups of modular group using Farey symbol.

Let $F_{\Gamma}$ be a fundamental domain of cofinite $\Gamma$ with finite number of sides and finite number of cusps. We partition $F_{\Gamma}$ as in \eqref{eq-partition-cuspidal-central} for some $T$ we choose later, 
$$F_{\Gamma}=F_T\bigcup_{\mathfrak{a} ~{\rm cusp}}F_{\mathfrak{a}, T}.  $$

1) First we show that, for  some $T$, the geodesic-covering number of $F_{\mathfrak{a}, T}$ in $\Gamma$ is finite for  every cusp $\mathfrak{a}$. In order to do this, we make use of Ford domains. 

For any cusp $\mathfrak{a}$, by \eqref{eq-conjugate-to-translation},  there exists $\sigma_{\mathfrak{a}}$ such that the stability group $\Gamma_{\mathfrak{a}}$ is generated by 
$$\sigma_{\mathfrak{a}}\begin{pmatrix}
1 & 1\\
0 & 1
\end{pmatrix}\sigma^{-1}_{\mathfrak{a}},$$
and the fundamental domain of $\sigma^{-1}_{\mathfrak{a}}\Gamma_{\mathfrak{a}}\sigma_{\mathfrak{a}}$ is 
\begin{equation}
P:=\{ z\in\ZH^2:  0\leq x< 1, y>0\}.
\end{equation}
Denote $\widetilde{\Gamma}^{\mathfrak{a}}:=\sigma^{-1}_{\mathfrak{a}}\Gamma\sigma_{\mathfrak{a}}$ and $$\widetilde{\Gamma}_{\infty}^{\mathfrak{a}}:=\sigma^{-1}_{\mathfrak{a}}\Gamma_{\mathfrak{a}}\sigma_{\mathfrak{a}}=\left\langle  \begin{pmatrix}
1 & 1\\
0 & 1
\end{pmatrix}  \right\rangle.$$

By \eqref{infinite-part-of-fundamental-domain} and \eqref{eq-F-cusp-to-F-infty},   the geodesic-covering number of $F_{\mathfrak{a}, T}$ in $\Gamma$ is the same as the geodesic-covering number of $\sigma^{-1}_{\mathfrak{a}}(F_{\mathfrak{a}, T}  )=P(T)$ in $\sigma^{-1}_{\mathfrak{a}}\Gamma\sigma_{\mathfrak{a}}$, i.e. 
\begin{equation}\label{eq-geo-cover-number-to-P}
K_{\Gamma}(F_{\mathfrak{a}, T})=K_{\widetilde{\Gamma}^{\mathfrak{a}}}(P(T)).
\end{equation}

We define a  domain associated with cusp $\mathfrak{a}$ as
\begin{align}\label{eq-Ford-domain-cusp}
\mathcal{D}_{\mathfrak{a}}:=& \{ z\in P:  {\rm Im}(\gamma z) <{\rm Im} (z),    \forall \gamma\in \widetilde{\Gamma}^{\mathfrak{a}}\setminus\widetilde{\Gamma}_{\infty}^{\mathfrak{a}} \}\nonumber\\
=&\{ z\in P: |cz+d|>1, \forall \begin{pmatrix}
* & *\\
c &  d
\end{pmatrix}\in  \widetilde{\Gamma}^{\mathfrak{a}}\setminus\widetilde{\Gamma}_{\infty}^{\mathfrak{a}}    \}
\end{align}
which is a \textit{Ford domain} and thus a fundamental domain of $\widetilde{\Gamma}^{\mathfrak{a}}$.  Note that $\sigma_{\mathfrak{a}}^{-1}(F_{\Gamma})$ may not be the same as $\mathcal{D}_{\mathfrak{a}}$.

We want to choose large enough $T$ such that $P(T)\subset \mathcal{D}_{\mathfrak{a}}$ for all cusp $\mathfrak{a}$.  Since $\mathcal{D}_{\mathfrak{a}}$ is a fundamental domain of $ \widetilde{\Gamma}^{\mathfrak{a}}  $, the boundary of  $\mathcal{D}_{\mathfrak{a}}$ consists of finite number of pieces from\textit{ isometric circles} of the form $|z+\frac{d}{c}|=\frac{1}{|c|}$ for some 
$$c\neq 0, \begin{pmatrix}
* & *\\
c & d
\end{pmatrix}\in \widetilde{\Gamma}^{\mathfrak{a}}.$$
Thus there is a largest radius among these isometric circles, say $\frac{1}{c_{\mathfrak{a}}}$, actually (see \cite{Iwaniec}, \S 2.6)  
\begin{equation}\label{eq-min-c}
c_{\mathfrak{a}}=\min\left\{ c>0:  \begin{pmatrix}
* & *\\
c & *
\end{pmatrix}\in\widetilde{\Gamma}^{\mathfrak{a}}\setminus\widetilde{\Gamma}_{\infty}^{\mathfrak{a}}    \right\}.
\end{equation}
For the fundamental domain $F_{\Gamma}$, there are only finite number of cusps, we choose any large enough
$$T\geq 100+10\max_{\mathfrak{a} ~{\rm cusp}} \tfrac{1}{c_{\mathfrak{a}}  },$$
then $P(T)=\sigma^{-1}_{\mathfrak{a}}(F_{\mathfrak{a}, T}  )\subset \mathcal{D}_{\mathfrak{a}}$ for every cusp $\mathfrak{a}$. 

For the above choice of $T$, we are ready to estimate  $K_{\widetilde{\Gamma}^{\mathfrak{a}}}(P(T))$ for any $\mathfrak{a}$.  Consider the set
\begin{equation}
\mathcal{A}:=\big\{ \gamma\in\widetilde{\Gamma}^{\mathfrak{a}}: d_{\ZH^2}(z_1, \gamma z_2)\leq d_{\ZH}(z_1, z_2), z_1, z_2\in P(T), {\rm Im}(z_1)\geq {\rm Im}(z_2)  \big\},
\end{equation}
which, by Definition \ref{defn-geodesic-covering-number},  is a geodesic cover of $P(T)$ in $\widetilde{\Gamma}^{\mathfrak{a}}$.  

For any two points $z_1=x_1+iy_1$ and $z_2=x_2+iy_2$ in $P(T)$ with $y_1\geq y_2$, the only possible isometries $\gamma$ from
$$ \widetilde{\Gamma}_{\infty}^{\mathfrak{a}}=\left\langle  \begin{pmatrix}
1 & 1\\
0 & 1
\end{pmatrix}  \right\rangle,$$
such that $d_{\ZH^2}(z_1, \gamma z_2)\leq d_{\ZH^2}(z_1, z_2)$ are 
\begin{equation}
\mathcal{T}:=\left\{  \begin{pmatrix}
1 & -1\\
0 & 1
\end{pmatrix}, 
\begin{pmatrix}
1 &0\\
0 & 1
\end{pmatrix},
\begin{pmatrix}
1 & 1\\
0 & 1
\end{pmatrix}  \right\}.
\end{equation}

If $\gamma\in \widetilde{\Gamma}^{\mathfrak{a}}\setminus\widetilde{\Gamma}_{\infty}^{\mathfrak{a}} $, by the construction of $\mathcal{D}_{\mathfrak{a}}$  and \eqref{eq-min-c}, we have
\begin{equation}
{\rm Im}(\gamma z_2)=\frac{y_2}{ (cx_2+d)^2+c^2 y_2^2}\leq \frac{1}{c^2 y_2}\leq \frac{1}{c^2_{\mathfrak{a}} y_2}.
\end{equation}
Since $y_2\geq T\geq 100+\frac{10}{c_{\mathfrak{a}}}$, we deduce that
\begin{equation}
{\rm Im}(\gamma z_2)\leq \frac{1}{100 c^2_{\mathfrak{a}}+10c_{\mathfrak{a}}}<\frac{1}{10c_{\mathfrak{a}}}.
\end{equation}
Denote $\gamma z_2=x_0+iy_0$ ($y_0<\frac{1}{10c_{\mathfrak{a}}}$), then by the hyperbolic distance formula, 
\begin{equation}\label{eq-hyperbolic-cosh-distance}
2\cosh(d_{\ZH^2}(z_1, z_2))=\frac{(x_1-x_2)^2+y_1^2+y_2^2}{y_1 y_2},
\end{equation}
and $|x_1-x_2|\leq 1$, $y_1\geq y_2\geq T\geq 100+\frac{10}{c_{\mathfrak{a}}}$, 
we derive
\begin{align}
&2\cosh(d_{\ZH^2}(z_1, \gamma z_2))
-2\cosh(d_{\ZH^2}(z_1, z_2)) \nonumber\\
&=\frac{(x_1-x_0)^2+y_1^2+y_0^2}{y_1 y_0 }- \frac{(x_1-x_2)^2+y_1^2+y_2^2}{ y_1 y_2}   \nonumber\\
&\geq \frac{y_1}{y_0}-\frac{1}{y_1 y_2}-\frac{y_1}{y_2}-\frac{y_2}{y_1} \nonumber\\
&\geq y_1\Big(\frac{1}{y_0}-\frac{1}{y_2}\Big)-\frac{1}{100^2}-1\nonumber\\
&\geq \frac{10}{c_{\mathfrak{a}}} \Big( 10c_{\mathfrak{a}}-\frac{c_{\mathfrak{a}}}{10}     \Big)-2=99-2>0.
\end{align}

Hence we  have $\mathcal{A}=\mathcal{T}$. We derive that the geodesic-covering number of $P(T)$ in $\widetilde{\Gamma}^{\mathfrak{a}}$ is $\leq 3$.  Since  our choice of $T$  works for all cusps, and by \eqref{eq-geo-cover-number-to-P}, we conclude that the geodesic-covering number of $F_{\mathfrak{a}, T}$ in $\Gamma$ is finite for all cusp $\mathfrak{a}$, precisely $K_{\Gamma}(F_{\mathfrak{a}, T})\leq 3$.

2) Now we bound the geodesic-covering number of the central part $F_T$ in $\Gamma$. 
Define the diameter of $F_T$ as
$$\diam(F_T):=\max_{p, q\in F_T} d_{\ZH^2}(p, q).$$
Since $F_T$ is bounded, the diameter $\diam(F_T)$ is finite. Pick any point $O$ inside $F_T$ which is not fixed by any element in $\Gamma$ except identity, then the set
$$\mathcal{B}:=\big\{ \gamma\in\Gamma: d_{\ZH^2}(O, \gamma(O))\leq 3\diam(F_T)   \big\}$$
is a geodesic cover of $F_T$ in $\Gamma$. Indeed, for any $\gamma\not\in \mathcal{B}$ and any two points $p, q\in F_T$,  by triangle inequality, we get
\begin{align}
d_{\ZH^2}(p, \gamma(q))&\geq d_{\ZH^2}(O, \gamma(O))-d_{\ZH^2}(p, O)-d_{\ZH^2}(\gamma(O), \gamma(q))\nonumber\\
&\geq 3\diam(F_T)-\diam(F_T)-\diam(F_T)=\diam(F_T)\geq d_{\ZH^2}(p, q).
\end{align}

Since a Fuchsian group $\Gamma$ acts properly discontinuously on $\ZH^2$, the $\Gamma$ orbit of any point is locally finite. Thus the set $\mathcal{B}$ is finite. Therefore, the geodesic-covering number of $F_T$ in $\Gamma$ is finite.  The same proof also works for the closure of $F_T$. 
\end{proof}

\begin{remark}
	Explicitly counting the cardinality of a set of the type $\mathcal{B}$ is the so-called \textit{hyperbolic circle problem}, see e.g. Lax-Phillips \cite{Lax-Phillips} and Phillips-Rudnick \cite{Phillips-Rudnick} etc.
\end{remark}

\section{Finite index subgroups of the modular group}
In this section, we give the proof of Theorem \ref{thm-finite-index}. 

Let $\Gamma$ be a finite index subgroup of $\PSL(2, \ZZ)$ with $[\PSL(2,\ZZ):\Gamma]=\mu$. Let $F$ be a fundamental domain of $\PSL(2, \ZZ)$, we may choose 
$$F=\Big\{z\in\ZH^2: |\Re(z)|\leq \frac{1}{2}, |z|\geq 1 \Big\}.$$ 
If we have the right coset decomposition $$\PSL(2,\ZZ)=\bigcup_{i=1}^{\mu} \Gamma \alpha_i,$$ then 
\begin{equation}\label{finite-index-fundamental-domain}
F_{\Gamma}=\bigcup_{i=1}^{\mu} \alpha_i(F)
\end{equation}
 is a fundamental domain of $\Gamma$. One can choose the coset representatives properly to get a simply connected fundamental domain of $\Gamma$ (see \cite{Schoen}, Chapter IV, Theorem 3).
For example, for the principal congruence subgroup $$\Gamma(2)=\left\{ \gamma\in\PSL(2, \ZZ): \gamma\equiv \begin{pmatrix}
1 & 0\\
0 & 1
\end{pmatrix}  \bmod 2  \right\},$$
 with index $[\PSL(2, \ZZ):\Gamma(2)]=6$, see Figure \ref{figure-Gamma2} (the arrows show the side parings) for a fundamental domain of $\Gamma(2)$ and Figure \ref{figure-surface-Gamma2} the shape of the surface $\Gamma(2)\backslash\ZH^2$.

\begin{figure}[ht]
	\centering
	\includegraphics[width=0.5\textwidth]{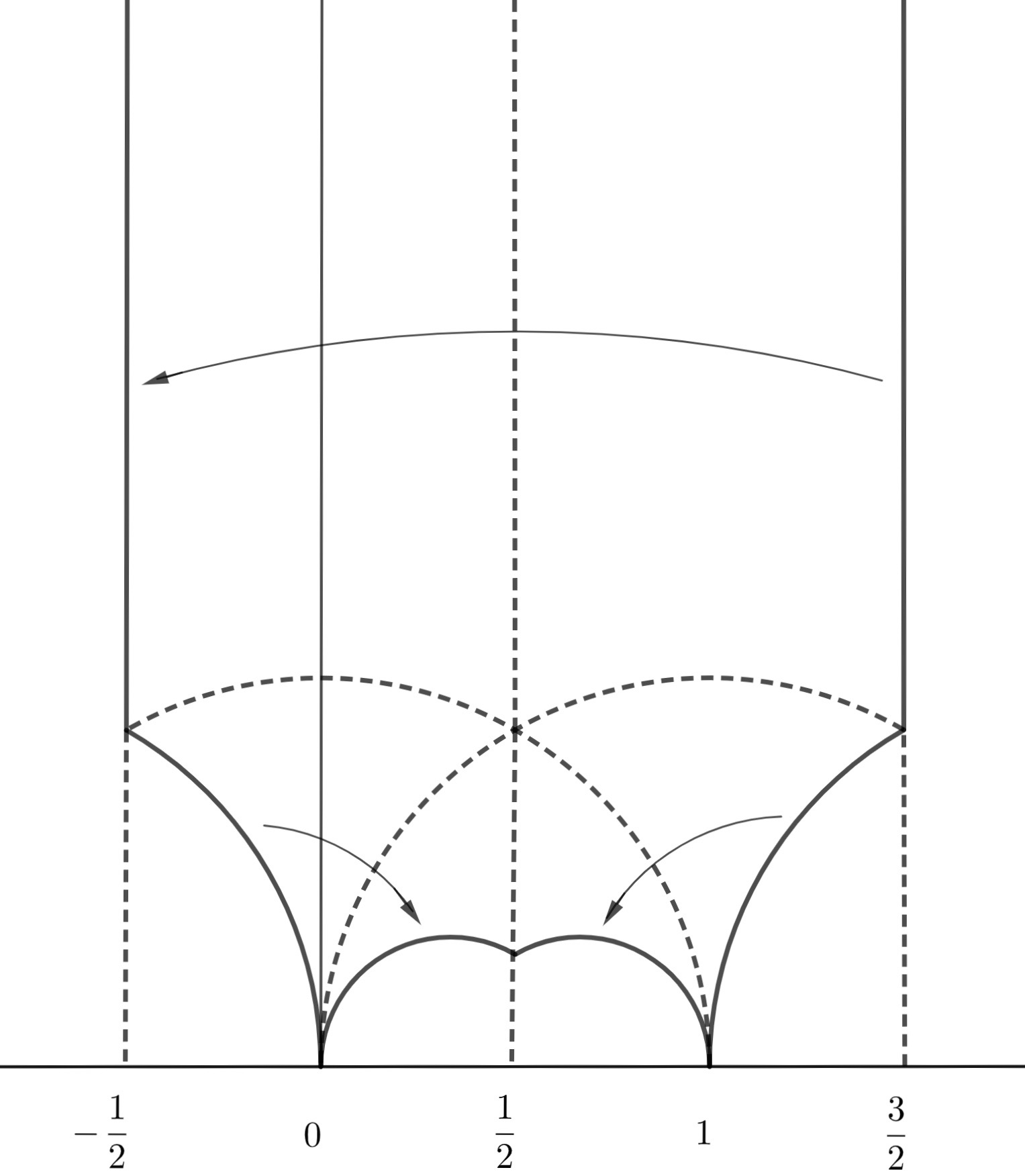}
\caption{Fundamental domain for $\Gamma(2) $}
	\label{figure-Gamma2}
	\end{figure}
\begin{figure}[ht]
	\centering
	\includegraphics[width=0.7\textwidth]{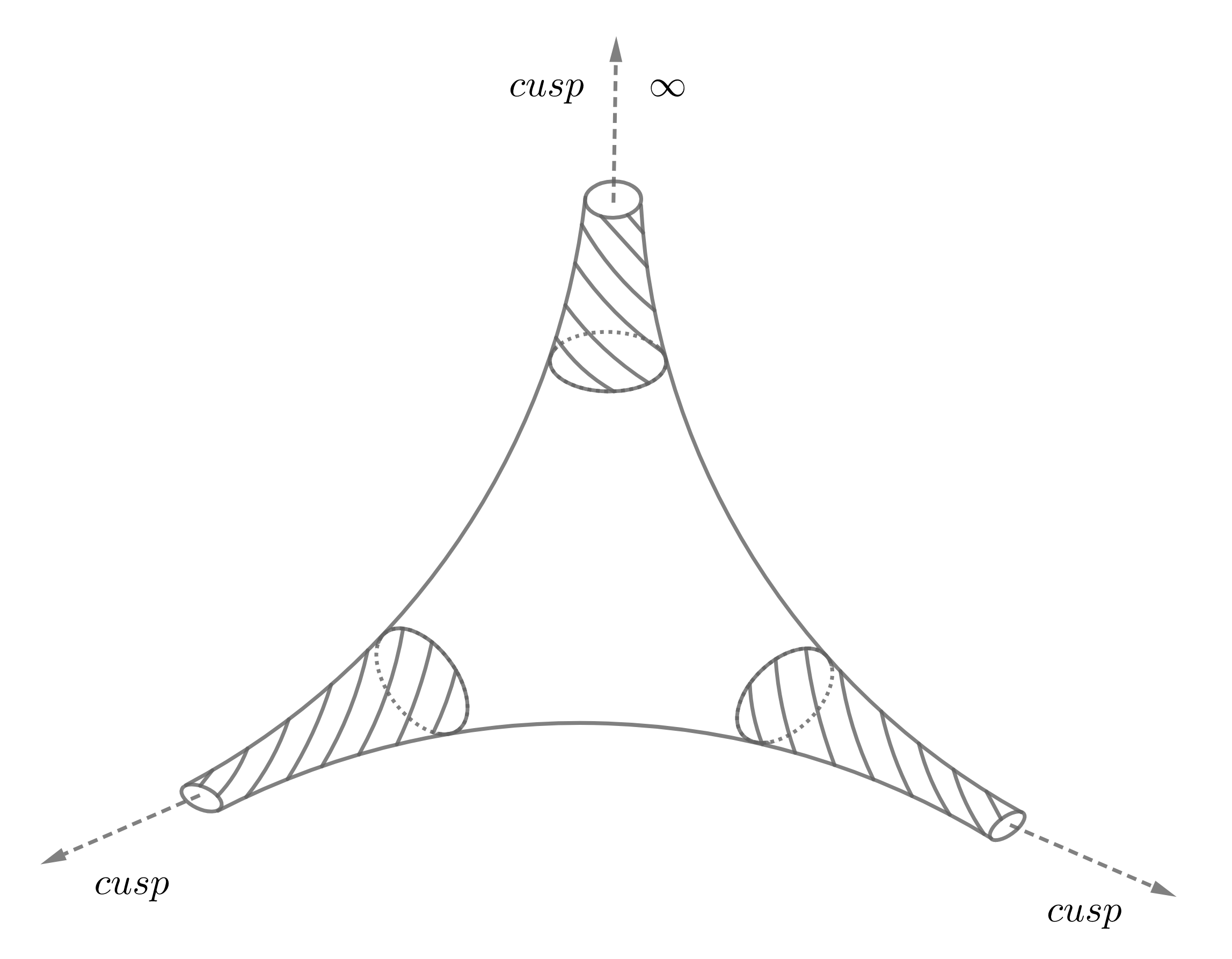}
	\caption{Shape of surface $\Gamma(2)\backslash\ZH^2$}
	\label{figure-surface-Gamma2}
\end{figure}

For a set $\mathcal{S}$ of $N$ points on $Y\cong\Gamma\backslash\ZH^2$, we consider their representatives in a fundamental domain $F_\Gamma$ constructed from the right coset decomposition.  Since $F_{\Gamma}$ is a union of $\mu$ copies of $F$, there exists an $\alpha_j$ such that $\alpha_j(F)$ contains $\geq N/\mu$ points from $\mathcal{S}$. Without loss of generality, we may assume $\alpha_j$ is identity and still denote this copy as $F$.  Otherwise, we just take $\alpha_j^{-1}(F_{\Gamma})$ as the fundamental domain of $\Gamma$ since $\alpha_j$ is an isometry of $\ZH^2$ and this transformation will not change distances and angles among the points we are considering. If we  have a lower bound for distinct distances among these $\geq N/\mu$ points, this would also give us a lower bound for distinct distances among all points of $\mathcal{S}$. 

We divide $F$ into two parts $F=F_{u}\cup F_o$ (see Figure \ref{figure-partition-fundamental-domain}) with 
\begin{equation}\label{one-fundamental-domain-parition}
F_u:=\Big\{ z=x+iy\in\ZH^2:  |x|\leq \frac{1}{2}, y\geq 2  \Big\} ~\text{and}~
F_o:=F\setminus F_u.
\end{equation}
\begin{figure}[ht]
	\centering
	\includegraphics[width=0.6\textwidth]{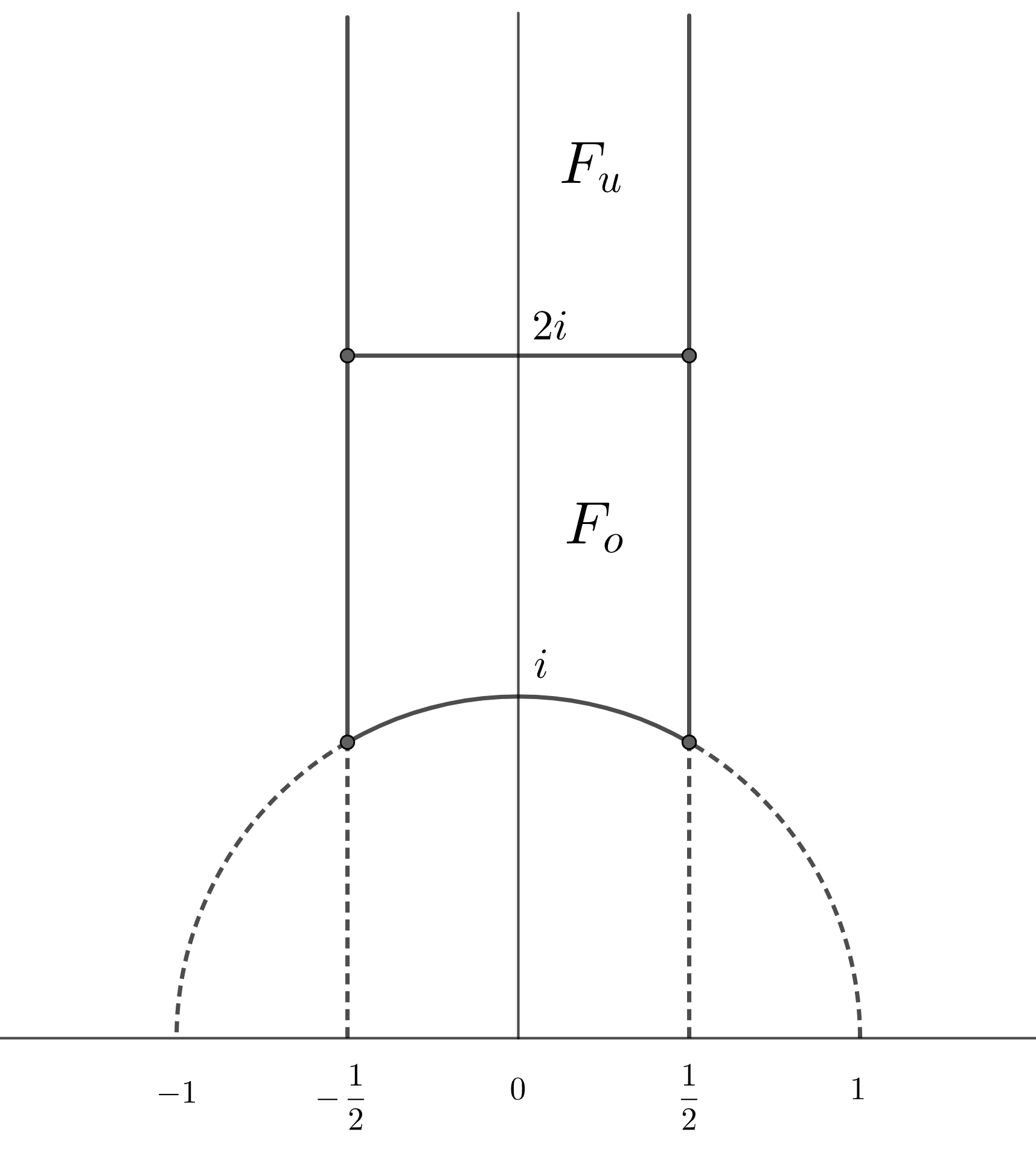}
	\caption{Partition of the fundamental domain $F=F_u+F_o$}
	\label{figure-partition-fundamental-domain}
\end{figure}

We want to bound the geodesic-covering numbers of $F_u$ and $F_o$ in different base groups. We prove the following lemma. 
\begin{lemma}\label{lem-finite-index-geo-cover-number}
For any subgroup $\Gamma$ of $\PSL(2, \ZZ)$, the geodesic-covering numbers $K_{\Gamma}(F_u)$ and $K_{\Gamma}(F_o)$ are  both bounded by some absolute constants. Precisely	
\begin{enumerate}[label=(\roman*)]
	\item\label{finite-index-cusp-part} The geodesic-covering number of $F_u$ in $\Gamma$ is $K_{\Gamma}(F_u)\leq 3$.
	\item\label{finite-index-central-part} The geodesic-covering number of $F_o$ in $\Gamma$ is  $K_{\Gamma}(F_o)\leq 252$.
\end{enumerate}
\end{lemma}
\begin{remark}
	The estimate in \ref{finite-index-central-part} may be improved by more careful calculations. We don't aim to optimize the constant here. The key point is that the geodesic-covering number of $F_o$ in any subgroup $\Gamma$ is absolutely bounded and thus independent of the index of $\Gamma$ in $\PSL(2, \ZZ)$. One may also use $y\geq U$ in the definition of $F_u$ for any large enough $U$ to optimize the estimate of $K_{\Gamma}(F_o)$. 
\end{remark}

Before giving the proof of Lemma \ref{lem-finite-index-geo-cover-number}, we use it to prove Theorem \ref{thm-finite-index} first. 
\begin{proof}[Proof of Theorem \ref{thm-finite-index}]
Suppose $\Gamma$ is a subgroup of $\PSL(2, \ZZ)$  of finite index  $[\PSL(2,\ZZ):\Gamma]=\mu$. 	
Let $\mathcal{S}$ be a set of $N$ points on the hyperbolic surface $Y\cong\Gamma\backslash\ZH^2$, and define the distance set
$$d_Y(\mathcal{S}):=\{ d_Y(p, q): p, q\in \mathcal{S}   \}.$$
If $F$ is a fundamental domain of $\PSL(2, \ZZ)$, by the fundamental domain of $\Gamma$ in the form \eqref{finite-index-fundamental-domain}, there exists some $j$ such that $\alpha_j(F)$ contains more than $N/\mu$ points. Since $\alpha_j$ is an isometry of $\ZH^2$, without loss of generality, we assume $\alpha_j(F)=F$ and let $\mathcal{S}_F$ be these $\geq N/\mu$ points on it. We observe that 
\begin{equation}\label{eq-distance-surface-to-one-F}
|d_Y(\mathcal{S})|\geq |d_Y( \mathcal{S}_F )|. 
\end{equation}
We use Lemma \ref{lem-finite-index-geo-cover-number} to establish a lower bound for $	|d_Y(\mathcal{S}_F )|$ and hence derive a lower bound for $|d_Y(\mathcal{S})|$. 

We partition the region $F=F_u\cup F_o$ as in \eqref{one-fundamental-domain-parition}.  Either $F_u$ or $F_o$ contains more than $\frac{1}{2}|\mathcal{S}_F|\geq N/2\mu$ points. 

\textbf{Case 1).}  The region $F_u$ contains more than $\frac{1}{2}|\mathcal{S}_F|$ points. Let $\mathcal{S}_u:=\mathcal{S}_F\cap F_u$ be the points on $F_u$, and $\Gamma_u$ be a geodesic-cover of $F_u$ in $\Gamma$ with cardinality $K_{\Gamma}(F_u)$. Then we have
\begin{align}
Q_Y(\mathcal{S}_u )&:=\big\{(p_1, p_2; p_3, p_4)\in \mathcal{S}_u^4: d_Y(p_1, p_2)=d_Y(p_3, p_4)\neq 0 )  \big\}\nonumber\\
&~\subset Q_{\ZH^2}\big(\cup_{\gamma\in\Gamma_u} \gamma(\mathcal{S}_u) \big),
\end{align}
where $Q_{\ZH^2}(\mathcal{P})$ is defined in \eqref{eq-defn-QH}. By Lemma \ref{lem-finite-index-geo-cover-number} \ref{finite-index-cusp-part} and \eqref{eq-Tao-quaruple}, we derive 
\begin{equation}
|Q_Y(S_u)|\ll K_{\Gamma}^3(F_u) |\mathcal{S}_u|^3 \log(K_{\Gamma}(F_u)|\mathcal{S}_u|)\leq 27 |\mathcal{S}_u|^3 \log(3|\mathcal{S}_u|).
\end{equation}
Consequently by \eqref{Quadruple-to-distance} and \eqref{eq-distance-surface-to-one-F}, we get the  lower bound
\begin{equation}
|d_Y(\mathcal{S})|\geq |d_Y(\mathcal{S}_u)|\gg \frac{|\mathcal{S}_u|}{ \log|\mathcal{S}_u| },
\end{equation}
where the implied constant is absolute. Therefore, by the assumption $\frac{N}{2\mu}\leq |\mathcal{S}_u|\leq N$,  we conclude that
\begin{equation}\label{eq-distance-finite-index-cuspidal}
|d_Y(\mathcal{S})|\geq C_1\frac{N}{\mu\log N}
\end{equation}
for some absolute constant 	$C_1>0$.

\textbf{Case 2).}	The region $F_o$ contains more than $\frac{1}{2}|\mathcal{S}_F|$ points. Let $\mathcal{S}_o=\mathcal{S}_F\cap F_o$ and $\Gamma_o$ be a geodesic cover of $F_o$ in $\Gamma$ with cardinality $K_{\Gamma}(F_o)$. By Lemma \ref{lem-finite-index-geo-cover-number} \ref{finite-index-central-part} and a similar argument as in \textbf{Case 1)}, we derive that
\begin{equation}
Q_Y(\mathcal{S}_o)\subset Q_{\ZH^2}\big( \cup_{\gamma\in\Gamma_o}\gamma(F_o)   \big)
\end{equation}
and thus
\begin{equation}
|Q_Y(\mathcal{S}_o)|\ll K_{\Gamma}^3(F_o) |\mathcal{S}_o|^3 \log(K_{\Gamma}(F_o)|\mathcal{S}_o|)\leq 252^3 |\mathcal{S}_o|^3 \log( 252 |\mathcal{S}_o|). 
\end{equation}	
Again by \eqref{Quadruple-to-distance} and the assumption $\frac{N}{2\mu}\leq |\mathcal{S}_o|\leq N$, we conclude that
\begin{equation}\label{eq-distance-finite-index-central}
|d_Y(\mathcal{S})|\geq |d_Y(\mathcal{S}_o)|\geq  C_2\frac{N}{\mu\log N}
\end{equation}	
for some absolute constant $C_2>0$. 

Finally, combining \eqref{eq-distance-finite-index-cuspidal} and \eqref{eq-distance-finite-index-central} and taking $C=\min\{C_1, C_2\}$, we get the desired lower bound for distinct distances in hyperbolic surfaces associated with any finite index subgroup of $\PSL(2, \ZZ)$, 
\begin{equation}
|d_Y(\mathcal{S})|\geq  C\frac{N}{\mu\log N}
\end{equation}
for some absolute constant $C>0$. 
\end{proof}

In the following we prove Lemma \ref{lem-finite-index-geo-cover-number}.

\begin{proof}[Proof of \ref{finite-index-cusp-part} in Lemma \ref{lem-finite-index-geo-cover-number}]
Recall that $F_u$ is the region 
$$\Big\{ z=x+iy\in\ZH^2:  |x|\leq \frac{1}{2}, y\geq 2  \Big\}.$$
We consider the set
\begin{equation}\label{eq-set-smaller-distance}
\mathcal{A}:=\{\gamma\in \PSL(2, \ZZ): d_{\ZH^2}(z_1, \gamma z_2)\leq d_{\ZH^2}(z_1, z_2), z_1, z_2\in F_u, {\rm Im}(z_1)\geq {\rm Im}(z_2) \},
\end{equation} 
which is a geodesic cover of $F_u$ in $\PSL(2, \ZZ)$ by Definition \ref{defn-geodesic-covering-number}.  For any subgroup $\Gamma$ of $\PSL(2, \ZZ)$, the set $\mathcal{A}\cap\Gamma$ is a geodesic cover of $F_u$ in $\Gamma$.  

For any two points $z_1=x_1+iy_1$ and $z_2=x_2+iy_2$ in $F_u$ with $y_1\geq y_2\geq 2$, and
 $$\gamma=\begin{pmatrix}
 a&b\\
 c&d
 \end{pmatrix}\in \PSL(2,\ZZ),$$
the imaginary part of $\gamma(z_2)$ can be written as
$$ \frac{y_2}{|cz_2+d|^2}=\frac{y_2}{(cx_2+d)^2+c^2 y_2^2 }.$$

If $c=0$, then $a=d=1$, the isometry  $\gamma$ is actually a translation of the form
$$\gamma=\begin{pmatrix}
1&b\\
0&1
\end{pmatrix}\in \PSL(2, \ZZ),$$
for some $b\in\ZZ$. The only possible choices of $\gamma$ for which $d_{\ZH^2}(z_1, \gamma z_2)\leq d_{\ZH^2}(z_1, z_2)$ are from the set
\begin{equation}\label{eq-translation-cover}
\mathcal{T}=\left\{  \begin{pmatrix}
1 & -1\\
0 & 1
\end{pmatrix}, 
\begin{pmatrix}
1 &0\\
0 & 1
\end{pmatrix},
\begin{pmatrix}
1 & 1\\
0 & 1
\end{pmatrix}  \right\}.
\end{equation} 

If $c\neq 0$, then $|c|\geq 1$ and thus
$${\rm Im}(\gamma z_2)\leq \frac{1}{y_2}\leq \frac{1}{2}.$$
Denote $\gamma(z_2)=x_0+iy_0$, then $y_0\leq \frac{1}{2}$. By the hyperbolic distance formula  \eqref{eq-hyperbolic-cosh-distance}
with the fact $y_1\geq y_2\geq 2$ and  $|x_1-x_2|\leq 1$, we get
\begin{align}
&2\cosh(d_{\ZH^2}(z_1, \gamma z_2))
-2\cosh(d_{\ZH^2}(z_1, z_2)) \nonumber\\
&=\frac{(x_1-x_0)^2+y_1^2+y_0^2}{y_1 y_0 }- \frac{(x_1-x_2)^2+y_1^2+y_2^2}{ y_1 y_2}   \nonumber\\
&\geq \frac{y_1}{y_0}-\frac{1}{y_1 y_2}-\frac{y_1}{y_2}-\frac{y_2}{y_1} \nonumber\\
&\geq 2y_1-\frac{1}{4}-\frac{y_1}{2}-1\geq \frac{7}{4}>0.
\end{align}
Thus for any  $\gamma\in \PSL(2, \ZZ)$ with $c\neq 0$, we always have $d_{\ZH^2}(z_1, \gamma z_2)>d_{\ZH^2}(z_1, z_2)$.  Hence $\mathcal{A}=\mathcal{T}$. 

For $\Gamma$ being any subgroup  of $\PSL(2, \ZZ)$, the elements of $\gamma'\in\Gamma$ such that $$d_{\ZH^2}(z_1, \gamma' z_2)\leq d_{\ZH^2}(z_1, z_2) ~\text{ with} z_1, z_2\in F_u, {\rm Im}(z_1)\geq {\rm Im}(z_2)$$ 
are also from the set $\mathcal{A}=\mathcal{T}$ in \eqref{eq-set-smaller-distance} and \eqref{eq-translation-cover}. Therefore, by Definition \ref{defn-geodesic-covering-number}, the set $\mathcal{T}\cap \Gamma$ is a geodesic cover of $F_u$ in $\Gamma$.  (Note that $\mathcal{T}\cap\Gamma$ always contains the identity.) We conclude that the geodesic-covering number of $F_u$ in any subgroup $\Gamma$ of $\PSL(2, \ZZ)$ is $K_{\Gamma}(F_u)\leq 3$. 
\end{proof}

\begin{proof}[Proof of \ref{finite-index-central-part} in Lemma \ref{lem-finite-index-geo-cover-number}]
Now we deal with the bounded part $$F_o=\{z=x+iy\in\ZH^2: |z|\geq 1, 0<y<2 \}.$$
We estimate the diameter of $F_o$,
\begin{align}\label{eq-Fo-diameter}
\cosh(\diam(F_o))&=\cosh\big(\max_{z_1, z_2\in F_o} d_{\ZH^2}(z_2, z_2)\big)\nonumber\\
&\leq \cosh\bigg(d_{\ZH^2}\Big(\frac{-1+\sqrt{3}i}{2}, \frac{1}{2}+2i\Big) \bigg)=\frac{23\sqrt{3}}{24}=1.6598\ldots
\end{align}
 Denote $r_0:=\max_{z\in F_o}d_{\ZH^2}(2i, z) $, then
\begin{equation}\label{eq-Fo-r0}
\cosh(r_0)=\cosh\bigg(d_{\ZH^2}\Big( 2i, \frac{1+\sqrt{3}i}{2} \Big)  \bigg)=\frac{5\sqrt{3}}{6}=1.4433\ldots
\end{equation}
The point $2i$ is not fixed by any element in $\PSL(2, \ZZ)$ except identity. By definition, the  set
\begin{equation}
\Gamma_o:=\{\gamma\in \PSL(2, \ZZ):   d_{\ZH^2}(2i, \gamma(2i))\leq \diam(F_o)+2r_0    \}.
\end{equation}
is a geodesic cover of $F_o$ in $\PSL(2, \ZZ)$. In fact, for any $\gamma\in\PSL(2, \ZZ)$ but not in $\Gamma_o$, we have 
\begin{equation}\label{eq-Fo-distance-increase}
d_{\ZH^2}(z_1, \gamma z_2)\geq \diam(F_o)\geq d_{\ZH^2}(z_1, z_2), \forall z_1, z_2\in F_o.
\end{equation}

Now we estimate the size of $\Gamma_o$. The set $\{\gamma(F_o): \gamma\in \Gamma_o \}$ is contained in the disc $\mathcal{D}(2i, R)$ centering at $2i$ of radius $R=\diam(F_o)+3r_0$.  Thus,
\begin{equation}
|\Gamma_o|\cdot {\rm Area}(F_o)={\rm Area}\Big(\bigcup_{\gamma\in\Gamma_o}\gamma(F_o)\Big)\leq {\rm Area}(\mathcal{D}(2i, R)).
\end{equation}
Since the area of the fundamental domain $F$ is $\pi/3$ and the area of $F_u$ is $1/2$, we derive that 
$${\rm Area}(F_o)=\frac{\pi}{3}-\frac{1}{2}=0.5471\ldots$$
By the hyperbolic area formula, \eqref{eq-Fo-diameter} and \eqref{eq-Fo-r0}, we get
$${\rm Area}(\mathcal{D}(2i, R))=2\pi(\cosh(R)-1)=\frac{\pi}{36}(848 +11\sqrt{4381})=137.5389\ldots
$$
Hence,
\begin{equation}
|\Gamma_o|\leq \frac{{\rm Area}(\mathcal{D}(2i, R))}{{\rm Area}(F_o)}   \leq 252.
\end{equation}

For any subgroup $\Gamma$ of $\PSL(2, \ZZ)$, by \eqref{eq-Fo-distance-increase}, we see that the set $\Gamma_o\cap\Gamma$ is a geodesic-cover of $F_o$ in $\Gamma$, and immediately we have $K_{\Gamma}(F_o)\leq |\Gamma_o|\leq 252$. 

\end{proof}

\end{document}